\documentclass{article}
\usepackage[utf8]{inputenc}
\usepackage{amsmath}
\usepackage{amssymb}
\usepackage{amsthm}
\usepackage{amsfonts, hyperref, accents}
\usepackage{diagbox}
\usepackage[all]{xy}
\usepackage{pgf,tikz}
\usepackage{lipsum}
\usetikzlibrary{arrows}
\usepackage{tikz}
\usepackage{tikz-cd}
\usepackage{xcolor}

\newcommand{\Princ}[1]{\mathrm{P}_{#1}}

\title{Non-Noetherian representation categories of generalized fields}
\author{Ma\"el Denys, M\'arton Hablicsek, Giacomo Negrisolo}
\newcommand{\Addresses}{{
  \bigskip
  \footnotesize

  M. Denys, \textsc{Department of Mathematics, Leiden University, Leiden, Netherlands,
  }\par\nopagebreak
  \textit{E-mail address}, 
  \texttt{maeldenys@outlook.com}

  \medskip

  M. Hablicsek, \textsc{Department of Mathematics, Leiden University, Leiden, Netherlands}\par\nopagebreak
  \textit{E-mail address}, \texttt{m.hablicsek@math.leidenuniv.nl}

  \medskip

  G. Negrisolo, \textsc{Department of Mathematics, Leiden University, Leiden, Netherlands}\par\nopagebreak
  \textit{E-mail address},  \texttt{gnegrisolo.gn@gmail.com}

}}

\date{\today}
\DeclareFontFamily{U}{rsf}{}
\DeclareFontShape{U}{rsf}{m}{n}{
  <5> <6> rsfs5 <7> <8> <9> rsfs7 <10-> rsfs10}{}
\DeclareMathAlphabet{\mathscr}{U}{rsf}{m}{n}

\DeclareMathAlphabet{\mathgth}{U}{euf}{m}{n}

\DeclareFontFamily{U}{cyr}{}
\DeclareFontShape{U}{cyr}{m}{n}{
  <5> wncyr5 <6> wncyr6 <7> wncyr7 <8> wncyr8 <9> wncyr9 <10-> wncyr10}{}
\DeclareMathAlphabet{\mathcyr}{U}{cyr}{m}{n}

\makeatletter
\def\operator@font{\sf}
\makeatother

\newcommand{\D}{{\mathbf D}}
\newcommand{\E}{{\mathbf E}}

\newcommand{\op}{{\mathsf{op}}}

\DeclareMathOperator{\Rep}{Rep}

\newcommand{\Hom}{{\mathsf{Hom}}}

\DeclareMathOperator{\Ob}{obj}

\DeclareMathOperator{\Mod}{Mod}
\DeclareMathOperator{\Vect}{Vect}

\newcommand{\ra}{\rightarrow}

\newcommand{\field}[1]{\mathbb{#1}}

\newcommand{\F}{\field{F}}
\newcommand{\B}{\field{B}}

\newcommand{\CC}{\mathcal{C}}
\newcommand{\DD}{\mathcal{D}}
\renewcommand{\phi}{\varphi}

\newcommand{\plus}{\mathop{\dot+}}
\newcommand{\pplus}{\mathop{\underaccent{\dot}{+}}}

\newtheorem{theorem}{Theorem}[section]
\newtheorem{proposition}[theorem]{Proposition}
\newtheorem{lemma}[theorem]{Lemma}

\newtheorem{corollary}[theorem]{Corollary}

\newtheorem{definition}[theorem]{Definition}

\newtheorem{Example}{Example}

\begin{document}

\maketitle

\abstract{In this paper, we show that the categories of finitely generated projective $\B$-modules and $\F_\infty$-modules with morphisms being (splittable) injections are not locally Noetherian. This provides another instance of the fact that these generalized fields have strange homological behavior.}
\section{Introduction}
The framework of \textbf{$FI$-modules} introduced by Church, Ellenberg and Farb \cite{church2015} and its generalization, the notion of \textbf{quasi-Gröbner Category}, introduced by Sam and Snowden \cite{samsno2017} provided a powerful framework to prove representation-theoretic results in various fields of mathematics. These results include the representation stability of cohomology groups of configuration spaces of points on an arbitrary (connected, oriented) manifold \cite{church2015, church2014}, the proof of the Lannes-Schwartz Artinian Conjecture \cite{putman2017, samsno2017}; and polynomial growth (on the number of edges) of the dimension of homology groups of configuration spaces of points on trees and graphs \cite{ProRam2019a, ProRam2019b}.

The Lannes-Schwartz Artinian Conjecture states (dually) that for any finite field, $\F_q$, the category of finitely generated vector spaces, $\Vect_{\F_q}$ over $\F_q$ with morphisms being splittable injections is \textbf{locally Noetherian}. It is natural to ask whether this statement extends to finite \textbf{generalized fields}. Generalized fields were introduced by Durov in his thesis \cite{durov2007} as a part of his new approach to Arakelov geometry \cite{arakelov1974, arakelov1975}. In this paper, we focus on two generalized fields, $\B$ and $\F_\infty$. In the theory of Durov, the generalized field, $\F_\infty$ is the \textbf{residue field at infinity}, while the Boolean semi-ring $\B$ is treated as a semi-field of characteristic one \cite{concon2019}. Even though basic algebro-geometric properties of the generalized fields, $\B$ and $\F_\infty$ resemble the properties of usual finite fields \cite{durov2007,joo2014,habjuh2018}, the homological properties are far from ordinary \cite{concon2019, sch2015}. 

In fact, in this paper, we show that the category of finitely generated projective $\B$-modules (or $\F_\infty$-modules) with morphisms being injections is not locally Noetherian illustrating another instance of the fact that generalized fields have strange homological behavior.

\section{$\B$-modules}

\begin{definition}
The Boolean semi-ring $\B$ is defined as the set $\{0,1\}$ equipped with two operations: a binary $\pplus$ and $\cdot$, with the following properties:
\begin{itemize}
\item $(a\pplus b)\pplus c=a\pplus (b\pplus c)$, $a\pplus b=b\pplus a$;
\item $0\pplus a=a\pplus 0=a$, hence $0$ is an \textbf{neutral element};
\item $a\pplus a=a$;
\item $(a\cdot b)\cdot c=a\cdot (b\cdot c)$, $a\cdot b=b\cdot a$;
\item $1\cdot a=a\cdot 1=a$.
\end{itemize}
\end{definition}
Given this algebraic structure, a $\B$-module is defined as follows.

\begin{definition}
\label{def:bmodule}
A \textbf{$\B$-module} is a structure $(V,0,-,\pplus)$ such that the following identities hold:
\begin{itemize}
\item $(a\pplus b)\pplus c=a\pplus (b\pplus c)$, $a\pplus b=b\pplus a$;
\item $a\pplus a=a$, $a\pplus 0=a$.
\end{itemize}
A notions of \textbf{submodule}, \textbf{congruence}, \textbf{quotient module} and \textbf{module homomorphism}, \textbf{finitely generated module} are defined as usual.
\end{definition}
Note that a $\B$-module is join-semilattice with a minimal element (0), the join of two elements is exactly $\pplus$.

We say that a $\B$-module $P$ is \textbf{projective}, if the usual lifting property holds: given any surjective morphism of $\B$-module $f:M\ra N$ and a $\B$-module morphism $g:P\ra N$, there exists a $\B$-module morphism $h:P\ra M$ so that $g=f\circ h$. It is not hard to show that a finitely generated $\B$-module $P$ is projective if and only if there exists a splittable surjection $F\rightarrow P$ from a free $\B$-module $F$ to $P$ (by splittable surjection, we mean a surjection which has a right inverse). In this way, the projective $\B$-modules can be equipped with a structure of a \textbf{finite distributive lattice} \cite{horn1971}. %)?

\begin{Example}\label{ex:b}
Consider the following finite distributive lattices of size $4n-3$ (see the figure below) denoted by $D_{n}$ (here $n>1$). The elements of the lattice are denoted by $a_{ik}$ representing the position of the element in the lattice. The operations (join and meet) $\vee$ and $\wedge$ are defined as follows:
\[a_{ik}\vee a_{lm}=a_{\max(i,l),\max(k,m)}, \quad a_{ik}\wedge a_{lm}=a_{\min(i,l),\min(k,m)}.\]
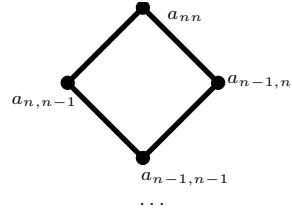
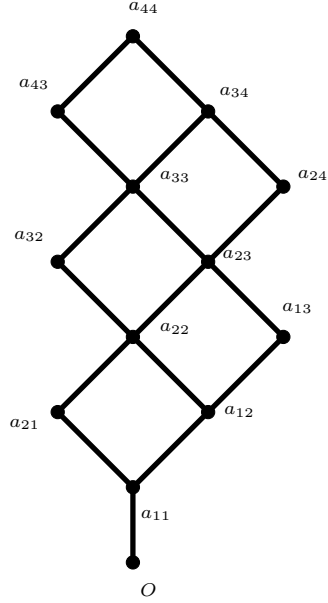
\begin{center}
\definecolor{ududff}{rgb}{0,0,0}
\definecolor{cqcqcq}{rgb}{0,0,0}
\begin{tikzpicture}[line cap=round,line join=round,>=triangle 45,x=.5cm,y=.5cm]
\clip(5.6721097057209438,4.741855581495774) rectangle (19.18008870615289,28.127014890185116);
\draw [line width=2.pt] (12.,28.)-- (10.,26.);
\draw [line width=2.pt] (10.,26.)-- (12.,24.);
\draw [line width=2.pt] (12.,24.)-- (14.,26.);
\draw [line width=2.pt] (14.,26.)-- (12.,28.);
\draw [line width=2.pt] (12.,20.)-- (10.,18.);
\draw [line width=2.pt] (10.,18.)-- (16.,12.);
\draw [line width=2.pt] (12.,20.)-- (16.,16.);
\draw [line width=2.pt] (14.,18.)-- (10.,14.);
\draw [line width=2.pt] (16.,16.)-- (10.,10.);
\draw [line width=2.pt] (16.,12.)-- (12.,8.);
\draw [line width=2.pt] (10.,14.)-- (14.,10.);
\draw [line width=2.pt] (10.,10.)-- (12.,8.);
\draw [line width=2.pt] (12.,6.)-- (12.,8.);
\begin{scriptsize}
\draw [fill=ududff] (12.,6.) circle (2.5pt);
\draw[color=ududff] (12.41700685660759,5.263624437166554) node {$O$};
\draw [fill=ududff] (12.,8.) circle (2.5pt);
\draw[color=ududff] (12.61700685660759,7.263624437166554) node {$a_{11}$};
\draw [fill=ududff] (14.,10.) circle (2.5pt);
\draw[color=ududff] (14.834755945813466,10.014645006989227) node {$a_{12}$};
\draw [fill=ududff] (12.,12.) circle (2.5pt);
\draw[color=ududff] (13.12560306398616,12.223797888816526) node {$a_{22}$};
\draw [fill=ududff] (10.,10.) circle (2.5pt);
\draw[color=ududff] (9.12430415693911,9.681187968222842) node {$a_{21}$};
\draw [fill=ududff] (14.,14.) circle (2.5pt);
\draw[color=ududff] (14.793073815967668,14.182857991569035) node {$a_{23}$};
\draw [fill=ududff] (16.,12.) circle (2.5pt);
\draw[color=ududff] (16.376994750108,12.765665576811902) node {$a_{13}$};
\draw [fill=ududff] (10.,14.) circle (2.5pt);
\draw[color=ududff] (9.249350546476505,14.683043549718613) node {$a_{32}$};
\draw [fill=ududff] (12.,16.) circle (2.5pt);
\draw[color=ududff] (13.12560306398616,16.266964483858942) node {$a_{33}$};
\draw [fill=ududff] (14.,18.) circle (2.5pt);
\draw[color=ududff] (14.709709556276072,18.51779949553204) node {$a_{34}$};
\draw [fill=ududff] (16.,16.) circle (2.5pt);
\draw[color=ududff] (16.793816048565983,16.308646613704738) node {$a_{24}$};
\draw [fill=ududff] (12.,20.) circle (2.5pt);
\draw[color=ududff] (12.292146025219775,20.768634507205135) node {$a_{44}$};
\draw [fill=ududff] (10.,18.) circle (2.5pt);
\draw[color=ududff] (9.3743969360139,18.72621014476103) node {$a_{43}$};
\draw [fill=ududff] (12.,24.) circle (2.5pt);
\draw[color=ududff] (13.12560306398616,23.444701466565206) node {$a_{n-1,n-1}$};
\draw [fill=ududff] (10.,26.) circle (2.5pt);
\draw[color=ududff] (9.3743969360139,25.47871517978032) node {$a_{n,n-1}$};
\draw [fill=ududff] (12.,28.) circle (2.5pt);
\draw[color=ududff] (13.12560306398616,27.77160343759837) node {$a_{nn}$};
\draw [fill=ududff] (14.,26.) circle (2.5pt);
\draw[color=ududff] (15.109709556276072,26.020582867775694) node {$a_{n-1,n}$};
\draw[color=black] (14.626345296584475,27.854596580990812) node {$$};
\draw[color=black] (14.626345296584475,24.895165361939146) node {$$};
\draw[color=black] (17.585776515636148,24.895165361939146) node {$$};
\draw[color=black] (17.585776515636148,27.854596580990812) node {$$};
\draw[color=black] (10.624860831387846,19.851627650597578) node {$$};
\draw[color=black] (12.62560306398616,14.891454198947603) node {$$};
\draw[color=black] (13.625974180285318,17.89256754784507) node {$$};
\draw[color=black] (11.625231947687004,16.850514301700116) node {$$};
\draw[color=black] (12.62560306398616,13.84940095280265) node {$$};
\draw[color=black] (13.625974180285318,10.848287603905188) node {$$};
\draw[color=black] (11.625231947687004,11.89034085005014) node {$$};
\draw[color=black] (10.624860831387846,8.889227501152678) node {$$};
\draw[color=ududff] (12.29288825781809,22.769376739803445) node {$\dots$};
\end{scriptsize}
\end{tikzpicture}
\end{center}
We remark that the $\B$-module structure is given by the join operation $\vee$.

It is easy to see that there exists a splittable surjection $F\ra D_n$. Indeed, the surjective map $g:\B[A_1,...,A_{2n-1}]\ra D_n$ from the free $\B$-module generated by $A_1$, $A_2$, ..., $A_{2n-1}$ given by
\[g(A_1)=a_{11}, g(A_2)=a_{12}, g(A_3)=a_{21},...,g(A_{2i})=a_{i-1,i+1},g(A_{2i+1})=a_{i+1,i}, ...\]
splits: the injection $h:D_n\rightarrow \B[A_1,...,A_{2n-1}]$ generated by
\[h(a_{11})=A_1, h(a_{12})=A_1+A_2, h(a_{21})=A_1+A_3 ...,\]
\[h(a_{i-1,i+1})=A_1\pplus A_2\pplus ...\pplus A_{2i-2}\pplus A_{2i},h(a_{i+1,i})=A_1\pplus A_2\pplus ...\pplus A_{2i-1}\pplus A_{2i+1}, \]
provides a right inverse.
\end{Example}

We have the following property for these distributive lattices.

\begin{proposition}\label{prop:bnonnot}
Let $D_n$ and $D_m$ be lattice as in Example \ref{ex:b}, with elements $a_{11}$, ..., $a_{nn}$, and $b_{11}$, ..., $b_{mm}$ respectively. Then, any injective $\B$-module morphism $f:D_n\ra D_m$ satisfying
\[f(a_{11})=b_{11}, f(a_{12})= b_{12}, f(a_{21})= b_{21}, f(a_{22})= b_{22}, \mbox{and}\]
\[f(a_{n-1,n-1})=b_{m-1,m-1}, f(a_{n-1,n})= b_{m-1,m}, f(a_{n, n-1})=b_{m, m-1}, f(a_{nn})= b_{mm}\]
is the identity morphism (and hence $n=m$). 
\end{proposition}

\begin{proof}
First, we show that $f(a_{13})=b_{13}$. Indeed, $a_{13}$ is the only element (other than $a_{11}$, $a_{12}$, $a_{21}$ and $a_{22}$) in the lattice which is not bigger than $a_{22}$. 

As a consequence, since $f$ is a $\B$-module map, we have to have that 
\[f(a_{23})=f(a_{22}+a_{13})=b_{22}+b_{13}=b_{23}.\] 
Now, we show that $f(a_{32})=b_{32}$. Indeed, $a_{32}$ is the only element (other than the previously discussed elements) which is not bigger than $a_{23}$. Hence, $f(a_{32})=b_{32}$. By continuing this process, we can show step by step that $f(a_{ij})=b_{ij}$ which concludes the proof.
\end{proof}

\begin{Example}\label{ex:bbmod}
Similarly, we consider the finite distributive lattice of 9 elements depicted in the figure below. We will denote this lattice by $\D_0$. Note that this is a sub-lattice of every $D_n$ for $n>3$.

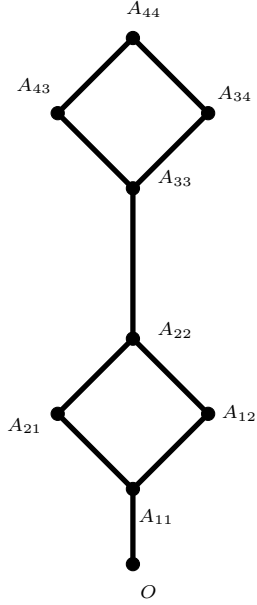
\begin{center}
\definecolor{ududff}{rgb}{0,0,0}
\definecolor{cqcqcq}{rgb}{0,0,0}
\begin{tikzpicture}[line cap=round,line join=round,>=triangle 45,x=.5cm,y=.5cm]
\clip(5.6721097057209438,4.741855581495774) rectangle (19.18008870615289,22.127014890185116);

\draw [line width=2.pt] (12.,20.)-- (10.,18.);
\draw [line width=2.pt] (10.,18.)-- (12.,16.);
\draw [line width=2.pt] (12.,20.)-- (14.,18.);
\draw [line width=2.pt] (14.,18.)-- (12.,16.);
\draw [line width=2.pt] (12.,16.)-- (12.,12.);
\draw [line width=2.pt] (12.,12.)-- (10.,10.);
\draw [line width=2.pt] (14.,10.)-- (12.,8.);
\draw [line width=2.pt] (12.,12.)-- (14.,10.);
\draw [line width=2.pt] (10.,10.)-- (12.,8.);
\draw [line width=2.pt] (12.,6.)-- (12.,8.);
\begin{scriptsize}
\draw [fill=ududff] (12.,6.) circle (2.5pt);
\draw[color=ududff] (12.41700685660759,5.263624437166554) node {$O$};
\draw [fill=ududff] (12.,8.) circle (2.5pt);
\draw[color=ududff] (12.61700685660759,7.263624437166554) node {$A_{11}$};
\draw [fill=ududff] (14.,10.) circle (2.5pt);
\draw[color=ududff] (14.834755945813466,10.014645006989227) node {$A_{12}$};
\draw [fill=ududff] (12.,12.) circle (2.5pt);
\draw[color=ududff] (13.12560306398616,12.223797888816526) node {$A_{22}$};
\draw [fill=ududff] (10.,10.) circle (2.5pt);
\draw[color=ududff] (9.12430415693911,9.681187968222842) node {$A_{21}$};

\draw [fill=ududff] (12.,16.) circle (2.5pt);
\draw[color=ududff] (13.12560306398616,16.266964483858942) node {$A_{33}$};
\draw [fill=ududff] (14.,18.) circle (2.5pt);
\draw[color=ududff] (14.709709556276072,18.51779949553204) node {$A_{34}$};

\draw [fill=ududff] (12.,20.) circle (2.5pt);
\draw[color=ududff] (12.292146025219775,20.768634507205135) node {$A_{44}$};
\draw [fill=ududff] (10.,18.) circle (2.5pt);
\draw[color=ududff] (9.3743969360139,18.72621014476103) node {$A_{43}$};
\draw[color=black] (14.626345296584475,27.854596580990812) node {$$};
\draw[color=black] (14.626345296584475,24.895165361939146) node {$$};
\draw[color=black] (17.585776515636148,24.895165361939146) node {$$};
\draw[color=black] (17.585776515636148,27.854596580990812) node {$$};
\draw[color=black] (10.624860831387846,19.851627650597578) node {$$};
\draw[color=black] (12.62560306398616,14.891454198947603) node {$$};
\draw[color=black] (13.625974180285318,17.89256754784507) node {$$};
\draw[color=black] (11.625231947687004,16.850514301700116) node {$$};
\draw[color=black] (12.62560306398616,13.84940095280265) node {$$};
\draw[color=black] (13.625974180285318,10.848287603905188) node {$$};
\draw[color=black] (11.625231947687004,11.89034085005014) node {$$};
\draw[color=black] (10.624860831387846,8.889227501152678) node {$$};
\draw[color=ududff] (12.29288825781809,22.769376739803445) node {$\dots$};
\end{scriptsize}
\end{tikzpicture}
\end{center}
\end{Example}

\begin{lemma}\label{lem:splinjb}
The injections $i:\D_0\rightarrow D_n$ for $n>3$ defined as 
\[i(A_{11})=a_{11}, i(A_{12})=a_{12}, i(A_{21})=a_{21}, i(A_{22})=a_{22},\]
\[i(A_{33})=a_{n-1,n-1}, i(A_{34})=a_{n-1,n}, i(A_{43})=a_{n,n-1}, i(A_{44})=a_{nn}\]
are splittable injections of $\B$-modules.
\end{lemma}

\begin{proof}
It is easy to see that the map $j:D_n\ra \D_0$ given by
\[j(a_{kl})=\begin{cases}
A_{kl} & k,l\leq 2,\\
A_{33} & \mbox{if }a_{12}<a_{kl}\leq a_{n-1,n-1}\mbox{ and } a_{kl}\ne a_{22}\\ 
A_{34} & a_{kl}=a_{n-1,n}\mbox{ or }a_{kl}=a_{n-2,n},\\
A_{43} & a_{kl}=a_{n,n-1},\\
A_{44} & a_{kl}=a_{nn}
\end{cases}\]
provides a left inverse.
\end{proof}

\section{\texorpdfstring{$\mathbb{F}_{\infty}$}{Foo}-modules}
%Maybe add the following statements: finitely presented projectives are free (hence Noeth), finitely generated projectives?
In this section, we introduce our motivating example, the category of finitely generated $\F_\infty$-modules. Our main reference for this section is $\cite{habjuh2018}$. 

We start with the definition of the generalized ring, $\F_\infty$ (see \cite{durov2007}, \textbf{5.1.16}).

\begin{definition}
The \textbf{generalized field} or \textbf{field} $\F_\infty$ is defined as the set $\{-1,0,1\}$ equipped with three operations: a unary $-$, and a binary $\plus$ and $\cdot$, with the following properties:
\begin{itemize}
\item $(a\plus b)\plus c=a\plus (b\plus c)$, $a\plus b=b\plus a$;
\item $0\plus a=a\plus 0=0$, hence $0$ is an \textbf{absorbing element};
\item $a\plus a=a$;
\item $a\plus(-a)=0$;
\item $-(1)=-1$, $-(-1)=1$, $-0=0$;
\item $(a\cdot b)\cdot c=a\cdot (b\cdot c)$, $a\cdot b=b\cdot a$;
\item $1\cdot a=a\cdot 1=a$;
\item $(-1)\cdot(-1)=1$.
\end{itemize}
\end{definition}

We remark that the operator $\plus$ is unlike the ordinary addition: the element $0$ is an absorbing element ($0\plus a=0$); and the operation $-$ is not the additive inverse, for instance, $(a\plus a)\plus (-a)$ is in general not $a$, but rather $a\plus (a\plus (-a))=a\plus 0=0$. 

This discussion leads us to the definition of a module over $\F_\infty$ (see \cite{durov2007}, \textbf{4.3.7}).

\begin{definition}
\label{def:module}
An \textbf{$\F_\infty$-module} is a structure $(V,0,-,\plus)$ such that:
\begin{itemize}
\item $(a\plus b)\plus c=a\plus (b\plus c)$, $a\plus b=b\plus a$;
\item $a\plus a=a$, $a\plus (-a)=0$;
\item $-(a\plus b)=(-a)\plus (-b)$; $-(-a)=a$.
\end{itemize}
Again, the notions of \textbf{submodule}, \textbf{congruence}, \textbf{quotient module}, \textbf{module homomorphism} and \textbf{finitely generated modules} are defined as usual.
\end{definition}

We remark that contrary to usual conventions, $0$ is not a neutral element, and $-a$ is not an additive inverse. In fact, an $\F_\infty$-module structure on a set $V$ gives rise to a natural partial order on $V$ given by $a\leq b$ if $a+b=b$; in this partial order $0$ is the \textbf{maximal} element.

We, again, consider projective $\F_\infty$-modules, modules with the usual lifting property. Similarly as with usual modules, a finitely generated $\F_\infty$-module $P$ is projective if and only if there exists a splittable surjection $F\rightarrow P$ from a free $\F_\infty$-module $F$ to $P$.

\begin{Example}\label{ex:f}
We modify Example \ref{ex:b} to obtain projective $\F_\infty$-modules. Let $D_n$ be the projective $\B$-module as in Example \ref{ex:b}. For each element $a_{ij}$ we add another element $-a_{ij}$. The addition $\plus$ is defined in an opposite fashion as before, explicitly, on the $a_{ij}$, we have
\[a_{ij}\plus a_{kl}=a_{\min(i,k), \min(j,l)},\]
on the $-a_{ij}$, we have
\[(-a_{ij})\plus (-a_{kl})=-a_{\min(i,k), \min(j,l)},\]
and finally
\[a_{ij}\plus (-a_{kl})=0.\]
In this fashion, we obtain an $\F_\infty$-module, which we call $E_n$ (the figure below illustrates the module $E_3$).

Again, there exists a surjection $g:\F_\infty[A_1,A_2,...,A_{2n-1}]\ra E_n$ generated by the images
\[g(A_1)=a_{nn}, g(A_2)=a_{n,n-1}, g(A_3)=a_{n-1,n},...\]
\[g(A_{2i})=a_{n-i,n-i+2}, g(A_{2i+1})=a_{n-i+1,n-i}.\]
This surjection splits, the map $h:E_n\ra \F_\infty[A_1,A_2,...,A_{2n-1}]$ given by the images
\[h(a_{nn})=A_1, h(a_{n,n-1})=A_1\plus A_2, h(a_{n-1,n})=A_1\plus A_3,...\]
\[h(a_{n-i,n-i+2})=A_1\plus A_2\plus ...\plus A_{2i-2}\plus A_{2i},h(a_{n-i+1,n-1})=A_1\plus ...\plus A_{2i-1}\plus A_{2i+1}\]
provides a right inverse.
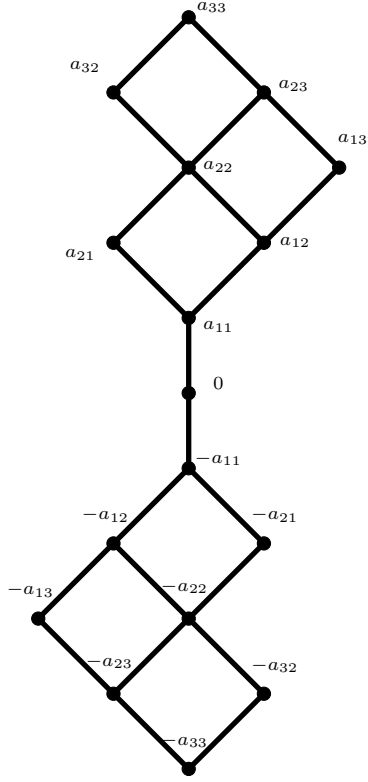
\begin{center}
    \definecolor{xdxdff}{rgb}{0,0,0}
\definecolor{ududff}{rgb}{0,0,0}
\definecolor{cqcqcq}{rgb}{0,0,0}
\begin{tikzpicture}[line cap=round,line join=round,>=triangle 45,x=0.5cm,y=0.5cm]
\clip(6.673223054618417,-5.137551424556686) rectangle (17.17897535725542,18.24760788413266);
\draw [line width=2.pt] (16.,12.)-- (12.,8.);
\draw [line width=2.pt] (10.,14.)-- (14.,10.);
\draw [line width=2.pt] (10.,10.)-- (12.,8.);
\draw [line width=2.pt] (12.,16.)-- (16.,12.);
\draw [line width=2.pt] (12.,16.)-- (10.,14.);
\draw [line width=2.pt] (10.,10.)-- (14.,14.);
\draw [line width=2.pt] (12.,8.)-- (12.,4.);
\draw [line width=2.pt] (12.,4.)-- (8.,0.);
\draw [line width=2.pt] (8.,0.)-- (12.,-4.);
\draw [line width=2.pt] (12.,4.)-- (14.,2.);
\draw [line width=2.pt] (10.,-2.)-- (14.,2.);
\draw [line width=2.pt] (10.,2.)-- (14.,-2.);
\draw [line width=2.pt] (12.,-4.)-- (14.,-2.);
\begin{scriptsize}
\draw [fill=ududff] (12.,8.) circle (2.5pt);
\draw[color=ududff] (12.79700685660759,7.763624437166554) node {$a_{11}$};
\draw [fill=ududff] (14.,10.) circle (2.5pt);
\draw[color=ududff] (14.834755945813466,10.014645006989227) node {$a_{12}$};
\draw [fill=ududff] (12.,12.) circle (2.5pt);
\draw[color=ududff] (12.79560306398616,12.023797888816526) node {$a_{22}$};
\draw [fill=ududff] (10.,10.) circle (2.5pt);
\draw[color=ududff] (9.124304156939111,9.681187968222842) node {$a_{21}$};
\draw [fill=ududff] (14.,14.) circle (2.5pt);
\draw[color=ududff] (14.793073815967668,14.182857991569035) node {$a_{23}$};
\draw [fill=ududff] (16.,12.) circle (2.5pt);
\draw[color=ududff] (16.376994750108,12.765665576811902) node {$a_{13}$};
\draw [fill=ududff] (10.,14.) circle (2.5pt);
\draw[color=ududff] (9.249350546476506,14.683043549718613) node {$a_{32}$};
\draw [fill=ududff] (12.,16.) circle (2.5pt);
\draw[color=ududff] (12.62560306398616,16.266964483858942) node {$a_{33}$};
\draw[color=black] (13.625974180285318,10.848287603905188) node {$$};
\draw[color=black] (11.625231947687004,11.89034085005014) node {$$};
\draw[color=black] (10.624860831387847,8.889227501152678) node {$$};
\draw[color=black] (13.625974180285318,13.89108308264845) node {$$};
\draw[color=black] (10.624860831387847,15.85014318540096) node {$$};
\draw[color=black] (12.583920934140362,11.89034085005014) node {$$};
\draw [fill=ududff] (12.,6.) circle (2.5pt);
\draw[color=ududff] (12.792146025219775,6.263438879016976) node {$0$};
\draw [fill=ududff] (12.,4.) circle (2.5pt);
\draw[color=ududff] (12.792146025219775,4.262696646418669) node {$-a_{11}$};
\draw [fill=ududff] (10.,2.) circle (2.5pt);
\draw[color=ududff] (9.79140379262146,2.76195441382036) node {$-a_{12}$};
\draw [fill=ududff] (14.,2.) circle (2.5pt);
\draw[color=ududff] (14.29288825781809,2.76195441382036) node {$-a_{21}$};
\draw [fill=xdxdff] (12.,0.) circle (2.5pt);
\draw[color=xdxdff] (11.892146025219775,0.8612121812220518) node {$-a_{22}$};
\draw [fill=xdxdff] (8.,0.) circle (2.5pt);
\draw[color=xdxdff] (7.790661560023146,0.7612121812220518) node {$-a_{13}$};
\draw [fill=ududff] (10.,-2.) circle (2.5pt);
\draw[color=ududff] (9.89140379262146,-1.1395300513762564) node {$-a_{23}$};
\draw [fill=ududff] (12.,-4.) circle (2.5pt);
\draw[color=ududff] (11.892146025219775,-3.240272283974565) node {$-a_{33}$};
\draw [fill=ududff] (14.,-2.) circle (2.5pt);
\draw[color=ududff] (14.29288825781809,-1.2395300513762564) node {$-a_{32}$};
\draw[color=black] (11.416821298458013,6.388299710404794) node {$$};
\draw[color=black] (9.62448971508869,2.845318673511956) node {$$};
\draw[color=black] (9.62448971508869,-2.1148547781380165) node {$$};
\draw[color=black] (12.62560306398616,2.8870008033577546) node {$$};
\draw[color=black] (12.583920934140362,-0.11411254553970801) node {$$};
\draw[color=black] (11.625231947687004,-0.11411254553970801) node {$$};
\draw[color=black] (13.584292050439519,-3.1152258944371702) node {$$};
\end{scriptsize}
\end{tikzpicture}
\end{center}
\end{Example}

As before, we have the following property.

\begin{proposition}\label{prop:fnonnot}
Let $E_n$ and $E_m$ be $\F_\infty$-modules as in Example \ref{ex:f}, with elements $a_{11}$, ..., $a_{nn}$, and $b_{11}$, ..., $b_{mm}$ respectively. Then, any injective $\F_\infty$-module morphism $f:E_n\ra E_m$ satisfying
\[f(a_{11})=b_{11}, f(a_{12})= b_{12}, f(a_{21})= b_{21}, f(a_{22})= b_{22}, \mbox{and}\]
\[f(a_{n-1,n-1})=b_{m-1,m-1}, f(a_{n-1,n})= b_{m-1,m}, f(a_{n, n-1})=b_{m, m-1}, f(a_{nn})= b_{mm}\]
is the identity morphism (and hence $n=m$). \qed
\end{proposition}

We omit the proof of the proposition above, it follows the same strategy as in Proposition \ref{prop:bnonnot}, but in this case, the recursive process begins with $a_{n-2,n}$ instead of $a_{13}$.

\begin{Example}
As in Example \ref{ex:bbmod} we define an $\F_\infty$-module of 17 elements depicted below. We denote this module by $\E_0$. For every $n>3$, we have an injection $\E_0\ra E_n$ defined as 
\[k(A_{11})=a_{11}, k(A_{12})=a_{12},k(A_{21})=a_{21}, k(A_{22})=a_{22},\]
\[k(A_{33})=a_{n-1,n-1}, k(A_{34})=a_{n-1,n}, k(A_{43})=a_{n,n-1}, k(A_{44})=a_{nn}\]
and $k(-A_{ij})=-k(A_{ij})$. This injection is a splittable injection as in Lemma \ref{lem:splinjb}, we omit the proof of this statement, it is completely parallel (reversed).

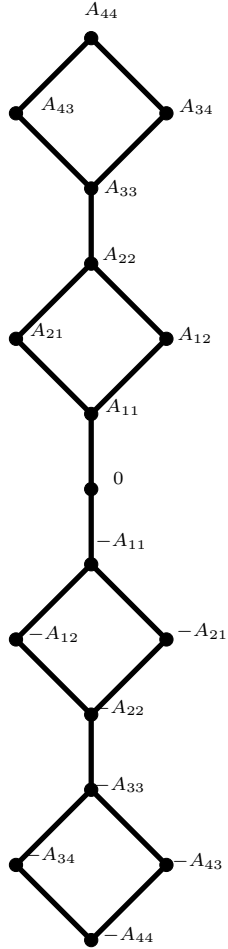
\begin{center}
\definecolor{xdxdff}{rgb}{0,0,0}
\definecolor{ududff}{rgb}{0,0,0}
\definecolor{cqcqcq}{rgb}{0,0,0}
\begin{tikzpicture}[line cap=round,line join=round,>=triangle 45,x=1.0cm,y=1.0cm]
\clip(-0.72,-4.3) rectangle (10,8.52);
\draw [line width=2.pt] (5.,8.)-- (4.,7.);
\draw [line width=2.pt] (4.,7.)-- (5.,6.);
\draw [line width=2.pt] (5.,8.)-- (6.,7.);
\draw [line width=2.pt] (6.,7.)-- (5.,6.);
\draw [line width=2.pt] (5.,6.)-- (5.,5.);
\draw [line width=2.pt] (5.,5.)-- (4.,4.);
\draw [line width=2.pt] (4.,4.)-- (5.,3.);
\draw [line width=2.pt] (5.,5.)-- (6.,4.);
\draw [line width=2.pt] (6.,4.)-- (5.,3.);
\draw [line width=2.pt] (5.,3.)-- (5.,1.);
\draw [line width=2.pt] (5.,1.)-- (4.,0.);
\draw [line width=2.pt] (4.,0.)-- (5.,-1.);
\draw [line width=2.pt] (5.,1.)-- (6.,0.);
\draw [line width=2.pt] (6.,0.)-- (5.,-1.);
\draw [line width=2.pt] (5.,-1.)-- (5.,-2.);
\draw [line width=2.pt] (5.,-2.)-- (4.,-3.);
\draw [line width=2.pt] (4.,-3.)-- (5.,-4.);
\draw [line width=2.pt] (5.,-4.)-- (6.,-3.);
\draw [line width=2.pt] (6.,-3.)-- (5.,-2.);
\begin{scriptsize}
\draw [fill=ududff] (4.,7.) circle (2.5pt);
\draw[color=ududff] (4.56,7.11) node {$A_{43}$};
\draw [fill=ududff] (5.,8.) circle (2.5pt);
\draw[color=ududff] (5.14,8.37) node {$A_{44}$};
\draw [fill=ududff] (6.,7.) circle (2.5pt);
\draw[color=ududff] (6.4,7.09) node {$A_{34}$};
\draw [fill=ududff] (5.,6.) circle (2.5pt);
\draw[color=ududff] (5.4,6.00) node {$A_{33}$};
\draw [fill=ududff] (5.,5.) circle (2.5pt);
\draw[color=ududff] (5.38,5.11) node {$A_{22}$};
\draw [fill=ududff] (4.,4.) circle (2.5pt);
\draw[color=ududff] (4.42,4.11) node {$A_{21}$};
\draw [fill=ududff] (5.,3.) circle (2.5pt);
\draw[color=ududff] (5.42,3.09) node {$A_{11}$};
\draw [fill=ududff] (6.,4.) circle (2.5pt);
\draw[color=ududff] (6.38,4.05) node {$A_{12}$};
\draw [fill=ududff] (5.,2.) circle (2.5pt);
\draw[color=ududff] (5.36,2.15) node {$0$};
\draw [fill=ududff] (5.,1.) circle (2.5pt);
\draw[color=ududff] (5.4,1.31) node {$-A_{11}$};
\draw [fill=xdxdff] (4.,0.) circle (2.5pt);
\draw[color=xdxdff] (4.5,0.07) node {$-A_{12}$};
\draw [fill=ududff] (6.,0.) circle (2.5pt);
\draw[color=ududff] (6.48,0.11) node {$-A_{21}$};
\draw [fill=ududff] (5.,-1.) circle (2.5pt);
\draw[color=ududff] (5.38,-0.91) node {$-A_{22}$};
\draw [fill=ududff] (5.,-2.) circle (2.5pt);
\draw[color=ududff] (5.38,-1.93) node {$-A_{33}$};
\draw [fill=ududff] (4.,-3.) circle (2.5pt);
\draw[color=ududff] (4.46,-2.87) node {$-A_{34}$};
\draw [fill=ududff] (6.,-3.) circle (2.5pt);
\draw[color=ududff] (6.42,-2.95) node {$-A_{43}$};
\draw [fill=ududff] (5.,-4.) circle (2.5pt);
\draw[color=ududff] (5.5,-3.93) node {$-A_{44}$};
\end{scriptsize}
\end{tikzpicture}
\end{center}
\end{Example}

\section{Representation categories}
We consider the categories, $F(R)$ (and $FI(R)$), of finite rank free $R$-modules with morphisms being injections (or splittable injections resp.) for a generalized ring. Similarly, let $P(R)$ (and $PI(R)$), denote the category of finitely generated projective $R$-modules with morphisms being injection (or splittable injection resp.) for a generalized ring. In this section, we show that the categories $F(\B)$, $F(\F_\infty)$, $FI(\B)$ and $FI(\F_\infty)$ are locally Noetherian, however, the categories $P(\B)$, $PI(\B)$, $P(\F_\infty)$ and $PI(\F_\infty)$ are not locally Noetherian. This gives another illustration of the strange homological behavior of the generalized rings $\B$ and $\F_\infty$. 
%The main examples are the category of finitely generated $\F_\infty$-modules with injections, $I(\F_\infty)$ and the opposite category of fintiely generated $\F_\infty$-modules with surjections, $S(\F_\infty)$. Other examples include the category $I(\Pos)$ (or $S(\Pos)^{op}$) of finite posets with injections (or with surjections resp.), the category $I(\Gp)$ (or $S(\Gp)^{op}$) of finite groups with injections (or with surjections resp.) and the category $I(\Rng)$ (or $S(\Rng)^{op}$) of finite rings with injections (or surjections resp.).

We start by recalling the basic definitions of \cite{samsno2017} and \cite{putman2017} concerning representation categories. Let $\CC$ be a category and $\Mod_k$ the category of (left) modules over a (left) Noetherian ring $k$. The \textbf{representation category}, $\Rep_k(\CC)$ is defined as the category whose objects are covariant functors $\CC\ra \Mod_k$ and whose morphisms are natural transformation between functors. Objects of the representation category $\Rep_k(\CC)$ are called \textbf{$\CC$-modules}. For instance, for every $X\in \Ob(\CC)$ we obtain a $\CC$-module (called the \textbf{principal project module}) defined as follows. For an object $Y\in \Ob(\CC)$, we assign the vector space generated by all maps $f:X\ra Y$ in $\CC$:
\[\Princ{X}(Y) = k[\Hom_{\mathcal{C}}(X,Y)]=\bigoplus_{f\colon X \rightarrow Y} k\cdot e_f\]
and for a morphism $g:Y\ra Z$ we assign the $k$-module morphism
\[k[\Hom_{\mathcal{C}}(X,Y)] \rightarrow k[\Hom_\CC(X,Z)]\colon e_f \mapsto e_{g\circ f}.\]
The representation category, $\Rep_k(\CC)$ is an Abelian category, with kernels, cokernels, images, direct sums, ... calculated point-wise. We say that a $\CC$-module is \textbf{finitely generated} if it is a quotient of a finite direct sum of $\Princ{X}$. This definition is equivalent to the definition given in the Introduction (see \cite{samsno2017}).

We say that the representation category $\Rep_{k}(\mathcal{C})$ is \textbf{Noetherian} if every submodule of a finitely generated $\CC$-module is also finitely generated. Similarly, we say that $\CC$ is \textbf{locally Noetherian} if the representation category $\Rep_k(\CC)$ is Noetherian for every (left-)Noetherian ring $k$.

\subsection{Negative results}
The following proposition shows that Property (G2) of \cite{samsno2017} is a necessary condition for a category to be locally Noetherian. For the sake of completeness, we provide the proof of this proposition. 

\begin{proposition}\label{prop: non noetherian}
Let $\mathcal{C}$ be an essentially small category. Assume $\mathcal{C}$ contains a fixed object $X_0$ and countably infinity distinct objects $Y_i$. If for each $i$ there exists a morphism $f_i \in \Hom_{\mathcal{C}}(X_0,Y_i)$ that does not factor as $X_0 \rightarrow Y_j \rightarrow Y_i$ for some $j < i$, then $\mathcal{C}$ is not locally Noetherian.
\end{proposition}

\begin{proof}
Consider the principal projective module $\Princ{X_0} \in \Rep_k(\mathcal{C})$ and the submodule $M$ generated by $\bigcup^{\infty}_{i =1} \Princ{X_0}(Y_i)$. Suppose $M$ is finitely generated by $\{\alpha_1, \ldots, \alpha_{\ell}\}$, where $\alpha_j \in M(X_j)$ for some object $X_j \in \mathcal{C}$. By the definition of $M$, each $\alpha_j$ is equal to a finite sum of the form  
\[\sum^{N_j}_{i=1} \sum_{g_{ij} \colon Y_i \rightarrow X_j} M(g_{ij})(\beta_{g_{ij}}),\]
where $\beta_{g_{ij}} \in \Princ{X_0}(Y_i)$. Hence, we conclude $M$ is also generated by $\bigcup^{N}_{i =1} \Princ{X_0}(Y_i)$, where $N = \max \{N_j \mid j \in \{1,2, \ldots, \ell\}\}$. However, by assumption $e_{f_{N+1}} \in \Princ{X_0}(Y_{N+1})$ can not be generated by these elements. We conclude that $M$ is not finitely generated and therefore that $\CC$ is not locally Noetherian.
\end{proof}

The proposition above implies that $P(\B)$, $PI(\B)$, $P(\F_\infty)$ and $PI(\F_\infty)$ are not locally Noetherian.

\begin{corollary} The categories
$P(\B)$, $PI(\B)$, $P(\F_\infty)$ and $PI(\F_\infty)$ are not locally Noetherian.
\end{corollary}

\begin{proof}
The remark below implies that it is enough to prove the statement for $P(\B)$ and $P(\F_\infty$). Finally, Proposition \ref{prop: non noetherian} combined with Proposition \ref{prop:bnonnot} and \ref{prop:fnonnot} imply that $P(\B)$ and $P(\F_\infty$) are not locally Noetherian (with $X_0$ being the projective module $\D_0$ or $\E_0$ respectively).
\end{proof}

\textbf{Remark:} If the category $\CC$ is a full subcategory of $\DD$, and $\CC$ satisfies the conditions of Proposition \ref{prop: non noetherian}, then $\DD$ also satisfies the conditions of Proposition \ref{prop: non noetherian} with the same objects. Similarly, if $\CC$ is a faithful subcategory of $\DD$ containing the objects $X_0$, the $Y_i$ and the morphisms $f_i:X_0\ra Y_i$ so that $\DD$ satisfies the conditions of Proposition \ref{prop: non noetherian}, then $\CC$ also satisfies the conditions of Proposition \ref{prop: non noetherian}. As an easy consequence, we obtain that the category of finite posets with injections and the category of finite lattices with injections are also not locally Noetherian (see, for instance, \cite{wiejar09}). Similarly, we obtain that the category of all finitely generated $\B$-modules (or $\F_\infty$-modules) with injection is also not locally Noetherian.

\subsection{Positive results}

We conclude the paper by showing that the category of free $\B$-modules (or $\F_\infty$-modules) of finite rank with (splittable) injections is locally Noetherian. We will only prove the statement for $\B$-modules, the same proof applies for $\F_\infty$-modules with minimal changes. The proof is similar to the proofs of Theorem 8.3.1. in \cite{samsno2017} and of Theorem C in \cite{putman2017}, we only highlight the key steps.

\begin{proposition}
The category of free $\B$-modules of finite rank with (splittable) injections is locally Noetherian.
\end{proposition}

\begin{proof}
%We first show the statement in the case of splittable injections. Note that if $f:F_n\ra F_m$ is a splittable injection between free $\B$-modules $F_n$ and $F_m$, then the generators of $F_n$ have to be mapped to generators. As a consequence, the functor $FI\ra FI(\B)$ from the category $FI$ of finite sets with injections to the category $FI(\B)$ sending a finite set to the $\B$-module generated by that set satisfies Property (F) of \cite{samsno2017} and is also essentially surjective. Since $FI$ is quasi-Gröbner, we have that $FI(\B)$ is locally Noetherian.

We first show the statement in the case of splittable injections. Let $FS$ denote the category of finite sets with morphisms being surjections. By Theorem 8.1.2 in \cite{samsno2017}, the opposite category $FS^{\op}$ is quasi-Gröbner. Consider the essentially surjective functor $\Phi\colon FS^{\op} \rightarrow F(\B) \colon S \mapsto \B[S]^* = \Hom_{\B}(\B[S],\B)$.
%, given on the level of morphisms by mapping the surjection $f\colon T \rightarrow S$ to the map
%\begin{align*}
%\Phi(f)\colon \B[S]^* \rightarrow \B[T]^* \colon g \mapsto (g \circ \B[f]), %\ 
%\text{ where } B[f] \colon\B[T] \rightarrow\B[S] \colon e_t \mapsto %e_{f(t)}.
%\end{align*}
Let $\B^n$ be a free $\B$-module and let $f \colon\B^n \rightarrow\B[S]^*$ be a splittable injection for a finite set $S$. 
%The image $T$, of $\{e_s \mid s \in S\}$ under the dual surjection $f^*\colon\B[S] \rightarrow (\B^n)^*$ generates $(\B^n)^*$. 
The dual surjection $f^*\colon\B[S] \rightarrow (\B^n)^*$ factorises as $\B[S] \rightarrow \B[T] \rightarrow (\B^n)^*,$ where the first map is induced by a surjection of finite sets. 
%Since $\B$ is finite, $(\B^n)^*$ has only finitely many spanning subsets. 
As a consequence, $\Phi$ satisfies Property (F) of \cite{samsno2017} implying that $FI(\B)$ is locally Noetherian.
%Maël: I rewrote it, but in fact it is just Theorem 8.3.1. in \cite{samsno2017}. And the second part is Theorem C in \cite{putman2017}. Maybe easier/better just to give the reference? Marci: Good point, I just made it shorter. Hopefully it is ok now. Maël: Looks good indeed, then I have no further remarks about the whole paper.
Now, we turn our attention to prove that $F(\B)$ is also locally Noetherian. Note that any map $f:\B^n\ra \B^m$ can be given by an $[m\times n]$ matrix, $A$, whose entries are $0$ or $1$. Let $l$ denote the number of distinct rows of this matrix. 
%Note, that $l\leq 2^n$ and we obtain a map $F_n\ra F_l$ restricting the $[m\times n]$ matrix, $A$, to the $l$ distinct rows. In this way, $F_l$ can be identified with the free module generated by the distinct rows of the matrix $A$. As a consequence, we obtain a factorization $F_n\ra F_l\ra F_m$ of $f$ where the latter map $F_l\ra F_m$ is given by sending a standard basis vector of $F_l$ corresponding to a distinct row to the sum of standard basis vectors of $F_m$ corresponding to the indices of this distinct row. Note, that this map is a splittable injection, the left inverse is given by mapping a standard basis vector $x$ of $F_m$ to the standard basis vector of $F_l$ corresponding to the row of $x$ in $A$.
%maybe rewrite? 
Note, that $l\leq 2^n$ and we obtain a map $\B^n\ra \B^l$ of $f$, given by the matrix $B$ consisting of the distinct rows of $A$. Moreover, we have a map $\B^l \ra \B^m$ sending the standard bases vector indexed by a row of $B$ to the sum of standard bases vector indexed by the equal rows in $A$. This map is a splittable injection and the two maps above yield a factorisation $\B^n\ra \B^l\ra \B^m$ of $f$. As a result, the forgetful functor $FI(\B)\ra F(\B)$ satisfies Property (F). Since this functor is also essentially surjective, we obtain that $F(\B)$ is also locally Noetherian.

%As a consequence, if $m>2^n$, then many rows of this matrix are the same, giving a factorization $F_n\ra F_l\ra F_m$ where the first map is the restriction of the $[m\times n]$ matrix to the $l$ distinct rows and the latter map is the injection .  where $l\leq 2^n$ and the latter map is a splittable injection. Therefore, the forgetful functor $FI(\B)\ra F(\B)$ has Property (F). Since this functor is also essentially surjective, we obtain that $F(\B)$ is also locally Noetherian.

%then apply Grobner with FI-modules. Finally, do the trick with a matrix to deal with non-splittable injections.
\end{proof}

\bibliographystyle{siam}
\bibliography{refs.bib}

\Addresses
\end{document}